\documentclass[
aps,%
11pt,%
final,%
notitlepage,%
oneside,%
onecolumn,%
nobibnotes,%
nofootinbib,%
superscriptaddress,%
noshowpacs,%
centertags]%
{revtex4}
\usepackage{amsmath, amsfonts, amssymb}
\usepackage{amsthm}

\newtheorem{lemma}{Lemma}
\newtheorem{theorem}[lemma]{Theorem}
\newtheorem{corollary}[lemma]{Corollary}

\newtheorem{definition}[lemma]{Definition}
\newtheorem{remark}[lemma]{Remark}
\usepackage{hyperref}
\begin{document}

\title{$E$ Type Singularities for Sufficiently Smooth Functions}

\author{\firstname{Ibrokhimbek}~\surname{Akramov}}

\email{i.akramov1@gmail.com}
\affiliation{%
Silk Road International University of Tourism and Cultural Heritage, Samarkand,
Uzbekistan
}%
\author{\firstname{Dildora}~\surname{Ikromova}}
\email{dikromova@mail.ru}
\affiliation{%
Department of Mathematics,
    Samarkand State University, Samarkand,
Uzbekistan
}%


\begin{abstract}
In this paper, we will consider $E$-type singularities which are Arnol'd type. We provide invariant conditions for a sufficiently smooth functions to have singularities of type $E_k (6\le k\le 8)$.   We show the functions can be reduced to $E_k, k=6, 7, 8$ type normal form under some certain conditions. Moreover, we show that result on normal form for  sufficiently smooth functions can be showed by the use of Implicit Function Theorem.  The results can be utilized to investigate oscillatory integrals with sufficiently smooth phase function.
\end{abstract}
\maketitle
\section{Introduction}

In this paper, we consider the functions $\phi$ having simple singularities. It is well established that \cite{agvMN82} the smooth functions having critical points with the multiplicity less than $8$ have only simple singularities.  To be more precise, for sufficiently smooth functions the notion of multiplicity cannot be defined by corresponding local algebra. However, a concept of  \textit{sufficient jet} of sufficiently smooth functions can be defined.

We give invariant definitions of singularities of type $E$ for sufficiently smooth functions having a critical point at the origin.
Moreover, we will provide some information about $E_k (6\le k\le 8)$-type singularities. The singularities in this case are of the Arnol'd type. Besides, one needs a classification of singularities, up to a linear transformation of variables. Because of some problems in harmonic analysis, such as the Fourier restriction problem and estimates for maximal operators related to hypersurfaces, as well as the summation of the Fourier transformation of surface-carried measures, which are invariant under a linear change of variables, (see \cite{IMmon16}, \cite{Imatzam} and also \cite{AkramovaDI20},  \cite{Imatzam}, \cite{tran}). As we shall see, in a linearly transformed coordinate system, the function has a special, simple form. However, this form contains four arbitrary, sufficiently smooth functions.  Nevertheless, the form is enough to investigate the behavior of the oscillatory integrals with this phase function.

 Analogical results could be obtained for a sufficiently smooth function. Then the results will be applied to investigate oscillatory integrals with sufficiently smooth functions. Such kind of oscillatory integrals arise in many mathematical-physics problems.

The paper is structured in the following order. 
In the next section, we present preliminary results on sufficiently smooth functions with singularities. The aim of Section 3 is to introduce invariant definitions for functions with $E_6$-type singularities. Section 4 focuses on smooth functions with $E_7$-type singularities, and we conclude with the theorem on the normal form  of sufficiently smooth functions having $E_8$ type of singularities.

\section{Preliminaries}

Suppose that $p(x_1, x_2)$ is a homogeneous polynomial of degree $k$, where $k \geq 0$.  We denote by $m_{S^1}(p_k)$ the largest order of zeros of a non-zero homogeneous polynomial $p_k$, where $S^1$ is the unit circle centered at the origin. For example, if $p_2(x_1, x_2) = x_1^2 - x_2^2$, then $m_{S^1}(p_2)=1$. If $p_2(x_1, x_2) = x_1^2 + x_2^2$, then $m_{S^1}(p_2)=0$. Whereas $m_{S^1}(p_2)=2$ for $p_2 = (x_1 - x_2)^2$.

Furthermore, we assume that $\phi(x_1, x_2)$ is a sufficiently smooth function and the condition $D^\alpha\phi(0) = 0$ is satisfied for any multi-index $\alpha = (\alpha_1, \alpha_2)\in \mathbb{Z}_+^2$  such that $|\alpha|:= \alpha_1 + \alpha_2 \leq 2$, where
\begin{eqnarray*}
\mathbb{Z}_+:= \mathbb{N}\cup \{0\}, \quad D^\alpha\phi(x)=\frac{\partial^{|\alpha|}\phi(x)}{\partial x_1^{\alpha_1}\partial x_2^{\alpha_2}}.
\end{eqnarray*}

Recall that any critical point of the continuously differentiable function $\phi$ is called a singular point of the function.
So the origin is a singular point of the function $\phi$. Certainly, any singular point can be moved to the origin with a simple translation.
Furthermore, we use the notation $C^k(U)$ for the set of all functions $\phi$ with domain $U$ having continuous mixed derivatives $D^\alpha\phi$ for any multi-index  $\alpha\in \mathbb{Z}_+^2 $ satisfying $|\alpha|\le k$.

Furthermore, we need to introduce a notion of multiplicity of the singularity, i.e. multiplicity of a critical point. The concept of the multiplicity of a critical point can take many different forms. Here, we define it by the local algebra of \textit{smooth germs}, so it is based on concepts of algebraic geometry.
To do this, we need some concepts from algebraic geometry (see \cite{agvMN82}).  Suppose $\phi$ is a $C^{\infty}$ function defined in a neighborhood of the origin of $\mathbb{R}^2$ and $\phi(0)=0,\,\nabla\phi(0)=0$. 
Naturally, these concepts can be formulated in the context of $\mathbb{R}^n$ or for the algebra of smooth functions defined on a finite-dimensional smooth manifold. To focus our discussion, we will limit ourselves to the two-dimensional scenario. Consequently, we will introduce some essential concepts from the algebra of functions that depend on two variables.
\begin{definition}\label{gradid}
The subring $I_{\nabla\phi} := \partial_1 \phi C^\infty(\mathbb{R}^2) + \partial_2 \phi C^\infty (\mathbb{R}^2)$, where $C^\infty(\mathbb{R}^2)$ is the algebra of equivalence classes of $C^\infty$ functions defined in a neighborhood of the origin, is called a gradient ideal.
 \end{definition}
Recall that two smooth functions, $f_1$ and $f_2$, defined in a neighborhood of the origin, are said to be equivalent if they agree in a neighborhood of the origin.  The set of all functions that are smooth and equivalent to a given function is called its equivalence class. These equivalence classes constitute a germ, or collection, of representatives.

\begin{definition}\label{localalg}
A local algebra of singularity is called to be the following quotient algebra:
\begin{eqnarray*}
 Q_{\phi}:=C^{\infty}(\mathbb{R}^2)/I_{\nabla\phi}.
 \end{eqnarray*}
 \end{definition}
 So, the  local algebra of singularity is an $\mathbb{R}-$ module or a linear space over field $\mathbb{R}$.
 \begin{definition}\label{multip}
 If this quotient  algebra has a finite $\mathbb{R}-$module structure, that is, a finite linear space with some dimension, then the number $dim_{\mathbb{R}}Q_{\phi}=k$ is called to be a multiplicity of the critical point $x = 0$ and it is said to be  a critical point with finite multiplicity. Otherwise, the critical point is called to be a  critical point of infinite multiplicity.
 \end{definition}
 For example, if $\phi(x_1,x_2)=x_1^2$, that is, it does not depend on the variable $x_2$, then $x = 0$ is the  critical point of infinite multiplicity.
 In fact, if the critical point is not isolated i.e. in any neighborhood of the critical point there exists another critical point different from that then multiplicity of that critical point is infinite. For  analytic functions in $\mathbb{C}^n$ the converse statement also holds. So, for a complex analytic function, the critical point has finite multiplicity if and only if the critical point is isolated \cite{agvMN82}. But, for real analytic function the converse statement does not hold, for example $\phi(x_1, x_2)= (x_1^2+x_2^2)^2$. It is easy to see that $(0, 0)$ is the isolated critical point. But, multiplicity of the critical point is infinite.
  Indeed, $p(x_1) \not\subset I_{\nabla\phi}$, for any non-trivial  polynomial $p$. Consequently, the system $\{1, x_1, \dots x_1^n\}$ is a linearly independent system of the local algebra $Q_{\phi}$. Hence $dim_{\mathbb{R}}Q_{\phi}=\infty$.

  It follows from the definition of the gradient ideal that for any $F:(\mathbb{R}^2,0)\mapsto(\mathbb{R}^2,0)$ (e.g $F$ is a smooth diffeomorph  map with $F(0)=0$)  the equations
$$
I_{\nabla\phi}=I_{\nabla(\phi\circ F)},\quad Q_{\phi}=Q_{\phi\circ F}
$$
are valid.

The multiplicity of a critical point does not change under a diffeomorphism. In particular, for smooth functions defined on a smooth manifold, the multiplicity of a critical point  does not depend on the choice of local coordinate system. Therefore, the notion of multiplicity is well-defined for finite-dimensional smooth manifolds.

\begin{definition}\label{defe7} (see \cite{agvMN82})
Let $\phi$ be a sufficiently smooth function defined in a neighborhood of the origin that satisfies the conditions: $\phi(0) = 0$ and the gradient of $\phi$, denoted by $\nabla \phi$, is zero at $0$.
The point $x = 0$ is said to be an $E_k (6\le k\le 8)$ type singularity if there exists a smooth map $\varphi \colon U \to \mathbb{R}^2$, 
such that $\varphi (U) \subseteq Dom(\phi)$, (where $U$ is a small neighborhood of $0$ and $Dom(\phi)$ is the domain of the function $\phi$), satisfying the conditions:
\begin{enumerate}
\item $\varphi (0) = 0$;
\item the Jacobian determinant of $\varphi$ at 0 is non-zero;
\item the following relation holds:
\begin{enumerate}
\item  $\quad \phi (\varphi_1 (y_1 , y_2) , \varphi_2 (y_1 , y_2 )) = \pm y^4_1 + y_2^3, \quad E_6\,$ type singularity;
\item $\quad \phi (\varphi_1 (y_1 , y_2) , \varphi_2 (y_1 , y_2 )) = y_2 y^3_1 + y_2^3, \quad E_7\,$ type singularity;
\item $\quad  \phi (\varphi_1 (y_1 , y_2) , \varphi_2 (y_1 , y_2 )) =  y^5_1 + y_2^3, \quad E_8\,$ type singularity.
\end{enumerate}
\end{enumerate}
The polynomial $\pm y^4_1 + y_2^3$, $y_2y_1^3 + y_2^3$, respectively  $ y^5_1 + y_2^3 $ is said to be  normal form for $E_6$,  $E_7$ respectively $E_8$ -type singularities.
\end{definition}

 If $\phi$ is a  $C^\infty$ smooth function defined in a neighborhood of the origin then we can see from the above, the multiplicity of the critical point is of the function $\phi$ having $E_k$ type singularity is $k=6, 7, 8$. Therefore, it is denoted by $E_k (6\le k\le 8)$, which means that the index of $E_k$ indicates the multiplicity of the critical point.

Although,  we cannot define a notion of the multiplicity of a critical point for sufficiently smooth functions.

If  we will restrict ourselves to non-degenerate linear transformations of variables. In this case, the normal forms will certainly contain smooth functions.

\begin{lemma}\label{eqsol}
Let $\phi$ 
be a $C^3(U)$ function in a neighborhood $U$ of the origin satisfying
\begin{eqnarray*}
\partial_2^2\phi(0, 0)=\partial_1\partial_2^2\phi(0, 0)=0, \quad \partial_2^3\phi(0, 0)\neq0.
\end{eqnarray*}
Then there exists a neighborhood $V\subset U$ and a $C^1$ function $\psi$ defined in $V$ such that $\psi(0)=\psi'(0)=0$ and 
\begin{eqnarray*}
\phi(x_1, x_2)=\phi(x_1, \psi(x_1))+\partial_{x_2}\phi(x_1, \psi(x_1))(x_2-\psi(x_1))+B(x_1, x_2)(x_2-\psi(x_1))^3,
\end{eqnarray*}
where $B$ is a continuous function with $B(0, 0)\neq0$.
\end{lemma}
\begin{proof}
A proof of Lemma \ref{eqsol} follows from the implicit function theorem and Taylor formula for $C^3$  function (see \cite{hor89} page no. 21).

If $\phi$ is a $C^\infty$ function satisfying the conditions of Lemma \ref{eqsol} then  according to the  implicit function theorem, the equation $\partial_2^2\phi(x_1, x_2)=0$  has a unique $C^\infty$ smooth solution $x_2=\psi(x_1)$  in sufficiently small neighborhood of the point $(0,0)$ in the form
 $$\psi(x_1)=x_1^m\omega(x_2),$$
 where $m\geq2$ and $\omega$ is a $C^\infty$ smooth function. We will assume $\omega(0)\not=0$, unless  the solution $\psi$ is  a flat function at the origin.  Otherwise  we can take the number $m\geq2$ as large as we wish and again $\omega$ is smooth, in addition, still the condition $\omega(0)=0$ is satisfied.
\end{proof}

\begin{corollary}\label{corolimp}
Let $\phi\in C^n(U) (n\ge3)$ be a function, where $U$ is a neighborhood of the origin, satisfying
\begin{eqnarray*}
\partial_2^2\phi(0, 0)=\partial_1\partial_2^2\phi(0, 0)=0, \quad \partial_2^3\phi(0, 0)\neq0.
\end{eqnarray*}
Then there exists a neighborhood $(-\delta_1, \delta_1)\times (\delta_2, \delta_2) \subset U$ and a $C^{n-2}$ function $\psi: (-\delta_1, \delta_2)\mapsto (-\delta_2, \delta_2)$ (where $\delta_1, \delta_2$ are sufficiently small fixed positive real numbers) such that $\psi(0)=\psi'(0)=0$ and  the following relation holds
\begin{eqnarray*}
\phi(x_1, x_2)=\phi(x_1, \psi(x_1))+\partial_{x_2}\phi(x_1, \psi(x_1))(x_2-\psi(x_1))+B(x_1, x_2)(x_2-\psi(x_1))^3,
\end{eqnarray*}
where $B\in C^{n-3}((-\delta_1, \delta_1)\times (\delta_2, \delta_2))$ and  $B(0, 0)\neq0$.
\end{corollary}

\section{Invariant definition of $E_6$ type singularities}

The purpose of this section is to discuss the criteria for a critical point of the sufficiently smooth function $\phi$  to have the $E_6$ type singularity, where we will state our one of the main results. 

\begin{theorem}\label{e6tsin} Let $\phi$ be a $C^7$ smooth function defined in a neighborhood of the origin, and let the Taylor  expansion of $\phi(x_1, x_2)$ up to order $6$ be given by:
\begin{equation*}
 \phi(x_1, x_2) \sim p_3(x_1, x_2) + p_4(x_1, x_2) + \dots.
 \end{equation*}
Then the following statements are equivalent:
\begin{enumerate}
\item[(i)] The function $\phi$ has $E_6$ type singularities at the origin;
\item[(ii)] The multiplicity of the root of $p_3$ is $3$ i.e. $m_{s^1}(p_3)=3$, and the resultant $Res(p_3, p_4)\not=0$.
\end{enumerate}
Where $Res(p_3, p_4)$ denotes the resultant of polynomials $p_3$ and $p_4$.
\end{theorem}

\begin{remark}
Note that $Res(p_3,p_4)\not=0$ if and only if the homogeneous polynomials $p_3,\, p_4$ have no common zeros out of the origin.
 \end{remark}

\begin{proof}
We prove that (ii) implies (i).

Indeed, by conditions (ii) it follows that the function $\phi$ after a 
potential linear change of variables can be described as
\begin{equation*}
 \phi(x_1, x_2)= x_2^3 + p_4(x_1, x_2) + R(x_1, x_2),
 \end{equation*}
where $R$ is a remainder term, and $p_4(x_1, x_2)=c_0x_1^4+c_1x_1^3x_2+c_2x_1^2x_2^2+c_3x_1x_2^3+c_4x_2^4$.
By our condition (ii) $Res(x_2^3, p_4)=c_0^3\neq0$. Further, we deploy the below linear change of variables:
\begin{equation*}
 y_2= x_2,\quad y_1=x_1 +\frac{c_1}{4c_0} x_2.
 \end{equation*}
and obtain:
\begin{equation*}
p_4\left(y_1-\frac{c_1}{4c_0}y_2, y_2\right)=c_0y_1^4+\tilde c_2y_1^2y_2^2+\tilde c_3y_1y_2^3+\tilde c_4y_2^4.
\end{equation*}

 So the function $\phi$ after a probable linear change of variables can be presented as:
 \begin{eqnarray*}
\phi(x_1, x_2)=x_2^3+c_0x_1^4+c_2x_1^2x_2^2+c_3x_1x_2^3+c_4 x_2^4 +R(x_1, x_2).
\end{eqnarray*}
We replace $y$ with $x$ and $\tilde{c}_j$ with $c_j$, where $j=2,3,4$, to simplify the notation.

Due to Corollary \ref{corolimp} we write the function $\phi$ in the form:
\begin{equation}\label{last6}
\phi(x_1, x_2)=b_0(x_1)+(x_2-\psi(x_1))b_1(x_1)+B(x_1, x_2)(x_2-\psi(x_1))^3,
\end{equation}
where $b_0$ is a $C^7$ function,  $b_1$ is a $C^5$ function, $B$ is a $C^4$ function with $B(0, 0)\neq0$ and $\psi$ is a $C^5$ function with $\psi(0)=\psi'(0)=0$.

The straightforward computations show that:
\begin{eqnarray*}
b_0(0)=b_0'(0)=b_0''(0)=b_0'''(0)=0\quad \mbox{and}\quad b_0^{(iv)}(0)=4!c_0\neq0,
\end{eqnarray*}
and also
\begin{eqnarray*}
b_1(0)=b_1'(0)=b_1''(0)=b_1'''(0)=0.
\end{eqnarray*}
Consequently, we have $b_0(x_1)=x_1^4\tilde b_0(x_1)$ with $\tilde b_0(0)=c_0\neq0$ and  $b_1(x_1)=x_1^4\tilde b_1(x_1)$ with $C^1$ smooth function $\tilde b_1$.
Therefore inserting the last expressions for $b_0$ and $b_1$  to \eqref{last6} we get
\begin{eqnarray*}
\phi(x_1, x_2)=x_1^4\tilde b_0(x_1)+(x_2-\psi(x_1))x_1^4\tilde b_1(x_1)+B(x_1, x_2)(x_2-\psi(x_1))^3,
\end{eqnarray*}

Consider the $C^1$ map defined in a neighborhood of the origin:
\begin{eqnarray*}
y_1=x_1\sqrt[4]{|\tilde b_0(x_1)+(x_2-\psi(x_1))\tilde b_1(x_1)|}\quad  y_2=(x_2-\psi(x_1))\sqrt[3]{B(x_1, x_2)}.
\end{eqnarray*}
Obviously, we have $J(0, 0) = \sqrt[4]{|\tilde b_0(0)|} \sqrt[3]{B(0,0)} \neq 0$, where $J(0, 0)$ is the determinant of the Jacobi matrix at the origin.
By the inverse map theorem in a sufficiently small neighborhood of the origin there exists the inverse $C^1$ map $x_1=\varphi_1(y_1, y_2), x_2=\varphi_2(y_1, y_2)$. Then, obviously, the following relation holds
\begin{eqnarray*}
\phi(\varphi_1(y_1, y_2), \varphi_2(y_1, y_2))=sign(\tilde b_0(0))y_1^4+y_2^3.
\end{eqnarray*}

Note that the relation $(i)\Rightarrow (ii)$ is easy to prove.

\end{proof}

From Theorem \ref{e78tsin} we conclude the following result.

\begin{corollary}
If the $\phi$ is a $C^7$ smooth function having a type $E_6$ singularity at the origin, then it is written by a linear transformation as follows:
$$
\phi(x_1,x_2)=b_3(x_1,x_2)(x_2-x_1^m\omega(x_1))^3+x_2x_1^4b_1(x_1)+x_1^4b_0(x_1)
$$
where $b_3,b_1,b_0,\omega$ are smooth functions, with $b_3(0,0)\neq0,\,b_0(0)\not=0$ and $m\geq2$. 
\end{corollary}

\section{The case  of $E_7$ type  singularities}

Here we will discuss the criteria for a critical point to be of the $E_7$ type. It is another crucial result of this paper.

\begin{theorem}\label{e78tsin}Let $\phi$ be a $C^9$ smooth function defined in a neighborhood of the origin, and let the Taylor development of $\phi(x_1, x_2)$ up to order $5$ be given by: \begin{equation*} \phi(x_1, x_2)= p_3(x_1, x_2) + p_4(x_1, x_2) + p_5(x_1, x_2)+R(x_1, x_2), \end{equation*}
where $R$ is the remainder term.
Then the following statements are equivalent:
\begin{enumerate}
\item[(i)] The function $\phi$ has type $E_7$ singularities at the origin;
\item[(ii)] The multiplicity of the root of $p_3$ is $3$, and the resultant $Res(p_3, p_4)=0$. The common linear factor of the homogeneous polynomials $p_3$ and $p_4$ is a simple linear factor for $p_4$.
\end{enumerate}
\end{theorem}

\begin{remark}
Note that $Res(p_3, p_4)=0$ if and only if the homogeneous polynomials $p_3,\, p_4$ have common zeros out of the origin.
 Therefore the Theorem \ref{e78tsin} can be formulated as following: It is necessary and sufficient for the function $\phi$ to have a singularity of type $E_7$ that the conditions $m_{S^1}(p_3) = 3$ and that the homogeneous polynomials $p_3$ and $p_4$ share a simple zero on the unit circle hold.
\end{remark}

{\sf Proof of Theorem \ref{e78tsin}.}
Actually, the assertion of the Theorem \ref{e78tsin} follows from Arnold's general singularities theory for $C^\infty$ smooth functions.
Note that due to the classical Tougeron Theorem  \cite{agvMN82} (also see \cite{tougeron})  in $C^\infty $ smooth case it is enough to prove analog of Theorem  \ref{e78tsin} for a polynomial function. Because any smooth $C^\infty$ function with a critical point of finite multiplicity can be transformed into a polynomial in a sufficiently small neighborhood of that critical point through a smooth change of coordinates.
But, we present an elementary and direct proof of Theorem \ref{e78tsin}. Suppose that the polynomials $m_{S^1}(p_3)=3$ and $p_3, p_4$  have  a common zero on the circle $S^1$. Since $m_{S^1}(p_3)=3$, then after a linear transformation the  function $\phi$ can be written as:
\begin{eqnarray*}\nonumber
\phi(x_1,x_2)=x_2^3+p_4(x_1,x_2)+p_5(x_1, x_2)+R(x_1, x_2)
\end{eqnarray*}
where
\begin{eqnarray*}\nonumber
p_4(x_1, x_2)=c_0x_1^4+c_1x_1^3x_2+c_2x_1^2x_2^2+c_3x_1x_2^3+c_4x_2^4
\end{eqnarray*}
and
\begin{eqnarray*}\nonumber
p_5(x_1, x_2)=b_0x_1^5+b_1x_1^4x_2+b_2x_1^3x_2^2+b_3x_1^2x_2^3+b_4x_1x_2^4+b_5x_2^5
\end{eqnarray*}
and $R$ is the remainder term.
 Since resultant of the polynomials are covariant of the linear group, hence the condition of Theorem \ref{e78tsin} can be written in the form $Res(p_3, p_4)=c_0^3=0$ so $c_0=0$. Since $(1, 0)$ is a simple root of $p_4$  on the unite circle $S^1$ then we have $c_1\neq0$.

First, we use change of variables  $x_1=y_1 , \, x_2=y_2-\frac{b_0}{c_1}y_1^2$. Then we obtain:
\begin{eqnarray*}\nonumber
\tilde\phi(y_1, y_2):= \phi\left(y_1, y_2-\frac{b_0}{c_1}y_1^2\right)=y_2^3+\tilde p_4(y_1, y_2)+\tilde p_5(y_1, y_2)+\tilde R(y_1, y_2),
\end{eqnarray*}
where
\begin{eqnarray*}\nonumber
p_4(y_1, y_2)=c_1y_1^3y_2+\tilde c_2y_1^2y_2^2+\tilde c_3y_1y_2^3+\tilde c_4y_2^4
\end{eqnarray*}
and
\begin{eqnarray*}\nonumber
p_5(y_1, y_2)=\tilde b_1y_1^4y_2+\tilde b_2y_1^3y_2^2+\tilde b_3y_1^2y_2^3+\tilde b_4y_1y_2^4+\tilde b_5y_2^5,
\end{eqnarray*}
and $\tilde R$ is the remainder term from the class $C^9$.
Note that we have the relations: $\partial_{y_1}^k\tilde \phi(0, 0)=0$ for $k=0, 1, \dots, 5$.

Now, we use the following linear change of variables $z_1=y_1+\frac{\tilde c_2}{3 c_1}y_2, \, z_2=y_2$ then we obtain:
\begin{eqnarray*}\nonumber
\phi_1(z_1, z_2):=\tilde\phi\left(z_1-\frac{\tilde c_2}{3 c_1}z_2,  z_2\right)=y_2^3+\tilde p_4\left(z_1-\frac{\tilde c_2}{3 c_1}z_2,  z_2\right)+\\  \tilde p_5\left(z_1-\frac{\tilde c_2}{3 c_1}z_2,  z_2\right)+\tilde R\left(z_1-\frac{\tilde c_2}{3 c_1}z_2,  z_2\right),
\end{eqnarray*}
where we have
\begin{eqnarray*}\nonumber
\tilde p_4\left(z_1-\frac{\tilde c_2}{3 c_1}z_2,  z_2\right)=c_1z_1^3z_2+\tilde c_3^1z_1z_2^3+\tilde c_4^1z_2^4
\end{eqnarray*}
and
\begin{eqnarray*}\nonumber
p_5\left(z_1-\frac{\tilde c_2}{3 c_1}z_2,  z_2\right)=\tilde b_1^1z_1^4z_2+\tilde b_2^1z_1^3z_2^2+\tilde b_3^1z_1^2z_2^3+\tilde b_4^1z_1z_2^4+\tilde b_5^1z_2^5.
\end{eqnarray*}

Since  $\phi_1\in C^9$, the equation $\partial_{z_2}^2\phi_1(z_1, z_2)=0$ has a $C^7$ solution $z_2=\psi(z_1)$ solution satisfying $\psi(0)=\psi'(0)=\psi''=0$.

In view of Corollary \ref{corolimp}, the function $\phi_1$ can be written as
$$
\phi_1(z_1, z_2)=b(z_1, z_2)(z_2-\psi(z_1))^3+(z_2-\psi(z_1))z_1^3b_1(z_1)+b_0(z_1),
$$
where $b$ is a $C^6$ smooth function with $b(0, 0)\neq0$,  $ b_1$ is a $C^7$  smooth function with $b_1(0)=c_1\not=0$ and $b_0$ is a $C^7$ smooth function.

It is clear that the following relations hold:
$
b_0(0)=b_0'(0)=\dots =b_0^{(5)}=0.
$
Consequently, the function $b_0$ can be written as $b_0(x_1)=x_1^6\tilde b_0(x_1)$ with a $C^1$ smooth function $\tilde b_0$.

\begin{lemma}\label{firstnorform}
If a $C^7$ smooth function has then form
\begin{equation}\label{bigdeg}
\phi(x_1, x_2)=b(x_1, x_2) x_2^3+x_2x_1^{3}b_1(x_1)+x_1^6 b_{0}(x_1)
\end{equation}
with $b, b_0, b_1 \in C^1$ functions with $b(0, 0)\neq0$ and $b_1(0)\neq0$ then there exists a $C^1$ smooth function $\varphi$ in a sufficiently small neighborhood of the origin $\mathbb{R}$  such that the the function
$\phi(x_1, x_2+x_1^3\varphi(x_1))$ can be written as
\begin{equation}\label{bigdeg}
\phi(x_1, x_2+x_1^3\varphi(x_1))=B(x_1, x_2) x_2^3+x_2x_1^{3}B_1(x_1, x_2),
\end{equation}
where $B, B_1$ are $C^1$ functions with $B(0, 0)\neq0, B_1(0, 0)\neq0$.
\end{lemma}

\begin{proof}
First, we try to find a solution to the equation $\phi(x_1, x_2)=0$ with respect to $x_2$ having the form $x_2=x_1^3\varphi(x_1)$ with a smooth function $\varphi$. For this reason, we use change of variables $x_2=x_1^3y_2$. Then we have
\begin{eqnarray*}
\phi(x_1, x_1^2y_2)=x_1^6(b(x_1, x_1^2y_2)y_2^3x_1^3+y_2b_1(x_1)+b_0(x_1))=0.
\end{eqnarray*}
Let us introduce the following auxiliary function
\begin{eqnarray*}
\phi_1(x_1, y_2)=b(x_1, x_1^2y_2)y_2^3x_1^3+y_2b_1(x_1)+b_0(x_1)
\end{eqnarray*}
and consider the equation $\phi_1(x_1, y_2)=0$ with respect to $y_2$ i.g. $y_2=y_2(x_1)$ satisfying $y_2(0)=-b_0(0)/b_1(0)$. Note that by our assumptions  $b_1(0)\neq0$. Moreover, we have $\phi_1(0, y_2(0))=0$ and $\partial_2\phi_1(0, -b_0(0)/b_1(0))\neq0$.
Thus we can use the classical implicit function Theorem and obtain a smooth solution  $y_2=y_2(x_1)$  of the equation $\phi_1(x_1, y_2)=0$ satisfying $y_2(0)=-b_0(0)/b_1(0)$.
Consequently, we have
$$\phi(x_1, x_1^3y_2(x_1))\equiv0$$
in a neighborhood $(-\delta, \delta)$ of the origin, where $\delta$ is a positive number.
By use of notation $\varphi(x_1):=x_2(x_1)$ and change of variables $y_2=x_2-x_1^3\varphi(x_1)$, then we obtain
\begin{eqnarray*}
\phi(x_1, y_2+x_1^3\varphi(x_1))=b(x_1, y_2+x_1^3\varphi)(y_2+x_1^3\varphi(x_1))^3+\\ (y_2+x_1^3\varphi(x_1))x_1^{3}b_1(x_1)+ x_1^6b_0(x_1).
\end{eqnarray*}

By our assumption $\phi(x_1, x_1^3\varphi(x_1))\equiv 0$.
On the other hand the function
$$B(x_1, y_2+x_1^3\varphi(x_1))(y_2+x_1^3\varphi(x_1))^3$$
can be written as
\begin{eqnarray*}
B(x_1, y_2+x_1^3\varphi(x_1))(y_2+x_1^3\varphi(x_1))^3=B(x_1, y_2+x_1^3\varphi(x_1))y_2^3+\\ y_2x_1^3\tilde B_1(x_1, y_2)+ B(x_1, x_1^3\varphi(x_1))x_1^9\varphi^3(x_1),
\end{eqnarray*}
where $\tilde B_1$ is a smooth function satisfying $\tilde B_1(0, 0)=0$.

Similarly the function $(y_2+x_1^3\varphi(x_1))x_1^{3}b_1(x_1)$ can be written as
\begin{eqnarray*}
(y_2+x_1^3\varphi(x_1))x_1^{3}b_1(x_1)= x_1^6\varphi(x_1)b_1(x_1)+y_2x_1^3b_1(x_1).
\end{eqnarray*}
Note that
\begin{eqnarray*}
\phi(x_1, x_1^3\varphi(x_1))= B(x_1, x_1^3\varphi(x_1))x_1^9\varphi^3(x_1)+x_1^6\varphi(x_1)b_1(x_1)+b_0(x_1)x_1^6=0
\end{eqnarray*}
for any $x_1\in (-\delta, \delta)$. Thus we have
\begin{eqnarray*}
\phi(x_1, y_2+x_1^3\varphi(x_1))=B(x_1, y_2+x_1^3\varphi(x_1))y_2^3+\tilde B^1(x_1, y_2)y_2x_1^3,
\end{eqnarray*}
where $\tilde B^1$ is a $C^1$ smooth function with $\tilde B^1(0, 0)=b_1(0)\neq0$.
\end{proof}

Finally, we use the below change of variables
\begin{eqnarray*}
z_2=y_2\sqrt[3]{\tilde B(x_1, y_2)}
 \quad  z_1=x_1\sqrt[3]{\frac{\tilde B^1(x_1, y_2)}{\sqrt[3]{\tilde B(x_1, y_2)}}}.
\end{eqnarray*}
It is a diffeomorphic change of variables. In the coordinates $z$ we have
\begin{eqnarray*}
\tilde\phi(z_1, z_2)=z_2^3+z_2z_1^3.
\end{eqnarray*}
The last form is called to be a \textit{normal form} of  the function having $E_7$ type singularity.

To complete the proof of Theorem \ref{e78tsin}, we show the inverse statement: That is, if the $\phi$ function  has a type $E_7$ singularity at the origin, then we show that the above conditions are satisfied. First of all, we note that if $\phi$ is a function of two variables, the rank of the Hessian at the critical point of this function (similar co-rank) does not depend on the diffeomorphic change of variables. Hence we come to the fact that $p_k=0$ for $0\leq k\leq2$. On the other hand, arbitrary diffeomorphic mapping acts as a linear mapping on a polynomial of $p_3$. Obviously, the linear mapping does not change $m_{S^1}(p_3)$. Therefore, if necessary, using possible linear change of variables, we can assume that $$p_3(x_1,x_2)=x_2^3.$$
Note that if $p_4(1, 0)\neq0$ then the function has $E_6$ type singularity (see Theorem \ref{e6tsin}).  All that remains is to show that $p_4(x_1, x_2)$ has the form:
\begin{eqnarray*}
p_4(x_1, x_2)=c_1x_1^3x_2+c_2x_1^2x_2^2+c_3x_1x_2^3+c_4x_2^4
\end{eqnarray*} with $c_1\not=0$. It means that $Res(p_3, p_4)=0$ and the common root is simple for $p_4$. Since $p_3(x_1, x_2)=x_2^3$, then surely the common root is $(1, 0)$ is real then
$$c_0=0 \quad \text{and}\quad c_1\neq0.$$
If $\varphi(y_1, y_2)=(\varphi_1(y_1, y_2), \varphi_2(y_1, y_2))$ is a $C^1$ smooth map satisfying the condition:
\begin{eqnarray*}
\phi(\varphi(y_1, y_2))=y_2^3+ y_2y_1^3,
\end{eqnarray*}
then the linear part of $\varphi$ has the form: $$(a_{11}y_1+a_{12}y_2, a_{22}y_2)\quad\text{with}\quad a_{11}\neq0,\quad a_{22}\neq0.$$
Then for the homogeneous part of degrees $3$ and  $4$ of $\phi$,  we have
$$p_3(y_1, y_2)=p_3(a_{11}y_1+a_{12}y_2, a_{22}y_2)=y_2^3$$ and  $$p_4(y_1, y_2)=p_4(a_{11}y_1+a_{12}y_2, a_{22}y_2)=y_2y_1^3$$ respectively. Consequently, we have $Res(p_3, p_4)=0$ and $p_3$, $p_4$ have a simple common root on $S^1$.  Theorem  \ref{e78tsin}  is proved. \qed

From Theorem \ref{e78tsin} we conclude the following result.

\begin{corollary}
If thesmooth function $\phi$ has a type $E_7$ singularity at the origin, then it is written by a linear transformation as follows:
$$
\phi(x_1,x_2)=b_3(x_1,x_2)(x_2-x_1^m\omega(x_1))^3+x_2x_1^3b_1(x_1)+x_1^{k_0}b_0(x_1)
$$
where $b_3,b_1,b_0,\omega$ are smooth functions, with $b_3(0,0)\neq0,\,b_1(0)\not=0$ and $m\geq2,\,k_0\geq5$. 
\end{corollary}

\section{Invariant definition of $E_8$ type singularities}

Aim of this section is to give another important result about the $E_8$ type singularities.
\begin{theorem}\label{e8tsin} Let $\phi$ be a $C^8$  smooth function defined in a neighborhood of the origin, and let the Taylor expansion of $\phi(x_1, x_2)$ be given by: $$\phi(x_1,x_2) = p_3(x_1,x_2) + p_4(x_1, x_2) +p_5(x_1, x_2)+R(x_1, x_2), $$ where $p_k (k=3, 4, 5)$ is a homogeneous polynomial function of degree  $k$. Then the following conditions are equivalents:
\begin{enumerate}
\item[(i)] The function $\phi$ has $E_8$ type singularities at the origin;
\item[(ii)]  $m_{S^1}(p_3) = 3$ and $Res(p_3, p_4) = 0$. $p_4$ shares a common root with $p_3$ on the unit circle $S^1$, with multiplicity at least two. Also, $Res(p_3, p_5) \neq 0$.
\end{enumerate}
\end{theorem}

\begin{remark}
Note that $Res(p_3,p_5)\not=0$ if and only if the homogeneous polynomials $p_3,\, p_5$ have no common zeros out of the origin as we mentioned before in this note.
Therefore Theorem \ref{e8tsin} can be formulated as following: It is a necessary condition to have a singularity of type $E_8$ at the origin for the function $\phi$ are  $m_{S^1}(p_3)=3$ and the homogeneous polynomials $p_3,\, p_5$ have no common zeros out of the origin. However, the condition $Res(p_3, p_5)\not=0$ is not sufficient. Because, the function  can have singularity of type $E_6$ or $E_7$. On the other hand we can give sufficient condition in terms of multiplicity of the critical point. The function has $E_8$ type singularity, if and only if $$p_2=0,\quad  m_{S^1}(p_3)=3\quad \text{and}\quad Res(p_3, p_5)\not=0,$$ the multiplicity of the critical point at the origin is $8$ as well.
\end{remark}

{\sf Proof of  Theorem \ref{e8tsin}.}
Actually, the assertion of Theorem \ref{e8tsin} follows from Arnold's general theory for $C^\infty$ functions.
But, we present an elementary and direct proof of Theorem \ref{e8tsin}. It turns out that the $E_8$ case is simpler than the $E_7$ case. Suppose that the polynomials $m_{S^1}(p_3)=3$.
We can  use a linear linear transform and denoting the new variables  by $x$ and, without loss of generality, the  function can written as:
\begin{eqnarray*}\nonumber
\phi(x_1,x_2)=x_2^3+p_4(x_1, x_2)+p_5(x_1, x_2)+R(x_1, x_2),
\end{eqnarray*}
where
 \begin{eqnarray*}\nonumber
p_4(x_1, x_2)=c_0x_1^4+c_1x_1^3x_2+c_2x_1^2x_2^2+c_3x_1x_2^3+c_4x_2^4,\\
p_5(x_1, x_2)=d_0x_1^5+d_1x_1^4x_2+d_2x_1^3x_2^2+d_3x_1^2x_2^3+d_4x_1x_2^4+d_5x_2^5.
\end{eqnarray*}
  Since the resultant of the polynomials is a covariant of the complete linear group, hence the condition of Theorem \ref{e8tsin} can be written in the form $c_0=c_1=0$.
Besides, if $c_2\neq0$, by using the change of variables $x_2+\frac{c_2}3x_1^2\to x_2$ (note that the last change of variables does not effect to the coefficient $d_0$), we come to the follow expression 
$$p_4(x_1, x_2)= c_3x_1x_2^3+c_4x_2^4.$$
Then in addition $d_0\not=0$. If necessary we can make a linear substitution of the form $x_1+d_1/(5d_0)x_2\to x_1, \, x_2\to x_2$. Then  in the polynomial $p_5(x_1, x_2)$
 we can make the coefficient in front of $x_1^3x_2$ monomial equal to zero, that is, we have
\begin{equation}\label{p5form}
p_5(x_1, x_2)=d_0x_1^5+\tilde d_2x_1^3x_2^2+\tilde d_3x_1x_2^3+\tilde d_4x_1x_2^4+\tilde d_5 x_2^5.
\end{equation}

Thus, after a change of variables we may assume that
\begin{eqnarray*}\nonumber
\phi(x_1, x_2)=x_2^3+c_3x_1x_2^3+c_4x_2^4+p_5(x_1, x_2)+R(x_1, x_2),
\end{eqnarray*}
with $p_5$ having the form \eqref{p5form}.

Owing to Corollary \ref{corolimp} the function $\phi$ can be written as:
$$
    \phi(x_1, x_2) = b(x_1, x_2)(x_2 -  \psi(x_1))^3 +  (x_2 -  \psi(x_1))b_1(x_1) + x_1^5 b_0(x_1)
$$
where $\psi$ is a $C^6$ function with $\psi(0)=\psi'(0)=\psi''(0)=0$, $b$ is a $C^5$ smooth function with $b(0, 0)\neq0$
where $b_1$ is a $C^6$  smooth function and $b_0$ is a $C^3$ smooth function.
Note that
\begin{eqnarray*}
\partial_2 \phi(0, 0)=b_1(0)=0, \, 0=\partial_2\partial_1^l \phi(0, 0)= b_1^{(l)}(0)=0, \quad \mbox{for}\quad l=1, 2, 3, 4.
\end{eqnarray*}

Thus, we can write $b_1(x_1) = x_1^5 \tilde{b}_1(x_1)$, where $\tilde{b}_1$ is a $C^1$ smooth function.

Thus, we arrive at the following equality:
\begin{eqnarray*}
\phi(x_1,x_2)=b(x_1,x_2)(x_2-x_1^m\omega(x_1))^3+x_2x_1^5\tilde b_1(x_1)+x_1^5b_0(x_1).
\end{eqnarray*}

Finally, given that $k_1\geq5$, we take following mapping transformations
\begin{equation}\label{mapxy}
y_1=x_1(b_0(x_1))^{\frac{1}{5}}\sqrt[5]{1+(x_2-\psi(x_1))\frac{\tilde b_1(x_1)}{|b_0(x_1)|}},\quad y_2=\sqrt[3]{b(x_1,x_2)}(x_2-x_1^m\omega(x_1)).
\end{equation}
According to the inverse mapping theorem, this map has an inverse in a sufficiently small neighborhood of  the point $(0, 0)$ and the inverse is a $C^1$  differentiable. Indeed, for the Jacobian of the map we get:
\begin{eqnarray*}\nonumber
\det \frac{D(y_1, y_2)}{D(x_1, x_2)}(0, 0)= |b_0(0)|^{\frac{1}{5}}\sqrt[3]{b(0, 0)} \not=0.
\end{eqnarray*}

It is apparent to see that in the $(y_1,y_2)$ coordinates $\phi$ has the following form:
\begin{equation}\label{normalxy}
\phi(x_1(y_1,y_2),x_2(y_1,y_2))=y_2^3+y_1^5.
\end{equation}
The functions that appear as a result of the diffeomorphic mapping  \eqref{mapxy} are exactly the normal form of Arnold's $E_8$-type singularity.

Thus we have shown that $m_{S^1}(p_3)=3$ for the function $\phi$ (of course under the condition $p_k\equiv0,(k\leq2)$) if the polynomials $p_3, p_4$  have a common zero of multiplicity at least two on $S^1$, and the polynomials $p_3, p_5$  do not have a common zero on $S^1$  then the function $\phi$ has  the $E_8$ type singularity at the origin.

To complete the proof of Theorem \ref{e8tsin}, we have to show its inverse: That is, if the $\phi$ function  have a $E_8$ type singularity at the origin, then we show that the above conditions are satisfied.

First of all, we note that if $\phi$ is a function of two variables, the rank of the Hessian at the critical point of this function (similar co-rank) does not depend on the diffeomorphic change of variables. Hence we come to the fact that $p_k=0$ for $0\leq k\leq2$. On the other hand, an arbitrary diffeomorphic mapping acts as a linear mapping on a polynomial of $p_3$. A linear mapping obviously does not change $m_{S^1}(p_3)$. Therefore, if necessary using possible linear change of variables, without loss of generality, we  can assume that $$p_3(x_1,x_2)=x_2^3.$$  If $Res(p_3, p_4)\neq0$ then, by Theorem \ref{e6tsin}, we get a contradiction to condition of Theorem \ref{e8tsin} because the function $\phi$ would have $E_6$ type singularity.  On the other hand if $Res(p_3, p_4)=0$ and $p_3$ and $p_4$ have a simple root then the function $\phi$ has $E_7$ type singularity (see Theorem \ref{e78tsin}). Thus $p_3$ and $p_4$ have a common root of multiplicity $2$.

Now it remains to show that $p_5(x_1, 0)=d_0x_1^4$ and $d_0\not=0$. Assume $d_0=0$.

 Then, after a change of variables we may assume that
\begin{eqnarray*}\nonumber
\phi(x_1, x_2)=x_2^3+c_3x_1x_2^3+c_4x_2^4+p_5(x_1, x_2)+R(x_1, x_2),
\end{eqnarray*}
with $p_5$ having the form
\begin{eqnarray*}\nonumber
p_5(x_1, x_2)=d_1x_2x_1^4+d_3x_2^3x_1^2+d_4x_2^4x_1+d_5x_2^5.
\end{eqnarray*}

We now return to a proof of the Theorem \ref{e8tsin}. The proof is done by assuming the opposite. Suppose that $$p_3(x_1,x_2)=x_2^3\quad \text{and}\quad p_5(x_1,0)\equiv0.$$ Then again we take the solution $x_1^4\omega(x_1)$, which is the unique smooth solution of the equation $\partial_2^2\phi(x_1, x_2)=0$, and obtain:
$$
\phi(x_1,x_2)=b_3(x_1,x_2)(x_2-x_1^4\omega(x_1))^3+x_2x_1^{k_1}b_1(x_1)+x_1^{k_0}b_0(x_1),
$$
here $k_1\geq4$ and $k_0\geq6$. Or $x_2-x_1^4\omega(x_1)\mapsto x_2$ take the substitution
\begin{eqnarray*}
\phi(x_1,  x_2)=b_3(x_1,x_2)x_2^3+x_2x_1^{k_1}b_1(x_1)+x_1^{k_0}b_0(x_1),
\end{eqnarray*}
where $b_3(0,0)=1$ and $k_1\geq4$ and $k_0\geq6$.
It is a contradiction which completes the proof. \qed

\section*{Acknowledgments}
Authors would like to give thanks Prof. Dr. Ikromov for helpful hints and comments. 
\section*{CONFLICT OF INTEREST}
The authors declare that they have no conflicts of interest.




\begin{thebibliography}{99}



\bibitem{agvMN82}
Arnol'd V. I., Gusein-Zade S. M. , Varchenko V. N.,
\newblock "Singularities of differentiable maps. The classification of critical points, Cauchy and Wave Fronts".
\newblock ({\em Birkh\"{a}user, Boston-Basel-Stuttgart}, 1985).


\bibitem{agv98}{
Arnol'd, V.I.,  Gusein-Zade, S.M. and  Varchenko, A.N.,
\newblock {"Singularities of differentiable maps".  {Vol. II}},  Monodromy and asymptotics of integrals,
 {\em Monographs in Mathematics, 83}.
\newblock Birkh\"auser, Boston Inc., Boston, MA, 1988.}


\bibitem{AkramovaDI20}
D. I. Akramova and I. A. Ikromov
\newblock {"Randol Maximal Functions and the Integrability
of the Fourier Transform of Measures"}
\newblock {\em Mathematical Notes}, {\bf 109, No. 5}, 3--20 (2021).






\bibitem{Dui74}
Duistermaat J. J.,
\newblock "Oscillatory integrals, Lagrange immersions and unfoldings of singularities".
\newblock {\em Comm. Pure
Appl. Math.} {\bf 27 (2)}, 207-281 (1974).



\bibitem{hor89}
H\"ormander L.,
\newblock "The analysis of linear partial differential operators. I.
Distribution theory and Fourier analysis".
\newblock {\em  Second edition, Springer-Verlag}, 1989.



\bibitem{IMmon16}
Ikromov, I.\,A.,  M\"uller, D.,
\newblock  "Fourier restriction for hypersurfaces in three dimensions and Newton polyhedra".
  {\sl Annals of Mathematics Studies 194}, Princeton University Press, Princeton and Oxford; 260 pp, (2016).

\bibitem{Imatzam}
Ikromova D.I., Ikromov I.A.,
\newblock "On the Sharp Estimates for Convolution Operators with Oscillatory Kernel".
\newblock {\em Journal of Fourier Analysis and Applications}, {\bf 29}, 28--50, (2024).

\bibitem{tougeron} Tougeron J. C. 
\newblock "Ideaux de founctions differentiables". 
\newblock {\em Ann. Inst. Fourier}, {\bf 18, No.1}, 177--240, (1968).

\bibitem{tran}
Stefan Bushnenke, Spyridon Dendrinos, Isroil A. Ikromov, and Detlef M\"uller,
\newblock  "Estimates for maximal functions associated to hypersurfaces in $\mathbb{R}^3$  with height $h < 2$": Part I,
\newblock {\em Trans AMS},
{\bf 372, No. 2}, 1363--1406 (2019).


\end{thebibliography}
\end{document}